\documentclass[11pt]{article}

\usepackage{amsmath}
\usepackage{amssymb}
\usepackage{epsfig}
\usepackage{url}
\usepackage{array}
\usepackage{varwidth}

\newtheorem{theorem}{Theorem}[section]

\newtheorem{definition}[theorem]{Definition}

\DeclareMathOperator{\lk}{lk}
\DeclareMathOperator{\diam}{diam}

\date{\today 
}

\begin{document}

\title{Not all simplicial polytopes are weakly vertex-decomposable}
\author{{\Large{Jes\'us A. De Loera\footnote{Email: \texttt{deloera@math.ucdavis.edu}}~ and   Steven Klee\footnote{Email: \texttt{klee@math.ucdavis.edu}}}} \\
Department of Mathematics \\
University of California \\
Davis, CA 95616}
\maketitle


\begin{abstract}
In 1980 Provan and Billera defined the notion of weak $k$-decomposability for pure simplicial complexes. They showed the diameter of a weakly $k$-decomposable simplicial complex $\Delta$ is bounded above by a polynomial function of the number of $k$-faces in $\Delta$ and its dimension.  For weakly $0$-decomposable complexes, this bound is linear in the number of vertices and the dimension.  In this paper we exhibit the first examples of non-weakly $0$-decomposable simplicial polytopes.


\end{abstract}

\section{Introduction}

Due to its relevance to the theoretical performance of the simplex method for linear programming, a lot of effort has been invested in bounding the diameter of convex polyhedra \cite{miketoddsurvey2002}. The 1957 Hirsch conjecture for polytopes became one of the most important problems in combinatorial geometry.  For simple polytopes, the Hirsch conjecture stated that any two vertices in a simple $d$-polytope with $n$ facets can be connected by an edge path of length at most $n-d$.  We will work in the polar setting where the Hirsch conjecture for simplicial polytopes asserts that if $P$ is a simplicial $d$-polytope with $n$ facets, then any pair of facets in $P$ can be connected by a facet-ridge path of length at most $n-d$. The Hirsch conjecture remained open until 2010 when F. Santos constructed a $43$-dimensional counterexample to the Hirsch conjecture with $86$ vertices.    (see \cite{Santos-Hirsch-counterexample} and its improvement in \cite{hirschimproved}). 

Despite this great success the best known counterexamples to the Hirsch conjecture have diameter $(1+\epsilon)(n-d)$, while the best known upper bounds on diameter are quasi-exponential in $n$ and $d$ \cite{Kalai-Kleitman} or linear in fixed dimension but exponential in $d$ \cite{Barnette, Larman}.  Unfortunately today we do not even know whether there exists a polynomial bound on the diameter of a polytope in terms of its dimension and number of vertices (see \cite{jesussurvey2011} and references therein for more information). Thus studying diameters of simplicial spheres and polytopes is still the subject of a great interest.

In the 1980's, motivated by the famous Hirsch conjecture, Provan and Billera defined two notions of $k$-decomposability for pure simplicial complexes.  First, they showed any $0$-decomposable simplicial complex satisfies the linear diameter bound posed by the Hirsch conjecture. Moreover $0$-decomposability is strong enough to imply shellability of the complex. Unfortunately,  Lockeberg \cite{Lockeberg} gave an example of a simplicial $4$-polytope on $12$ vertices that is not vertex-decomposable.  Similarly, Santos' counterexample to the Hirsch conjecture also is not vertex-decomposable. In the same paper,  Provan and Billera also showed that for fixed $k$,  the diameter of a \emph{weakly} $k$-decomposable simplicial complex $\Delta$ is bounded above by a linear function of the number of $k$-faces in $\Delta$.  Thus one approach to proving a \emph{linear} upper bound on the diameter of a simplicial polytope in terms of its dimension and number of vertices would be to show that every simplicial polytope is weakly vertex-decomposable (see Sections 5, 6, and 8 in \cite{Klee-Kleinschmidt} for a discussion on decomposability and weak decomposability). Unfortunately, the purpose of this note is to  provide the first examples of simplicial polytopes that are not even weakly vertex-decomposable.

In Section~\ref{background}, we give the necessary definitions pertaining to simplicial complexes, (weak) $k$-decomposability, and transportation polytopes.  In Section~\ref{non-wvd-example}, we give our main results (Theorems~\ref{counterexample} and \ref{counterexample-even}) which provide transportation polytopes in all dimensions $d \geq 5$ whose polars are not weakly vertex-decomposable.


\section{Definitions and background} \label{background}

\subsection{(weak) $k$-decomposability of simplicial complexes}

We recall some basic facts about simplicial complexes, for more details see \cite{Stanley:Combinatorics:1996}.
A \textbf{simplicial complex} $\Delta$ on vertex set $V = V(\Delta)$ is a collection of subsets $F \subseteq V$, called \textbf{faces}, such that if $F \in \Delta$ and $G \subseteq F$, then $G \in \Delta$.  The dimension of a face $F \in \Delta$ is $\dim(F) = |F| - 1$ and the dimension of $\Delta$ is $\dim(\Delta) = \max\{\dim(F): F \in \Delta\}$.  A \textbf{facet} of $\Delta$ is a maximal face under inclusion.  We say $\Delta$ is \textbf{pure} if all of its facets have the same dimension.

The \textbf{link} of a face $F$ in a simplicial complex $\Delta$ is the subcomplex $\lk_{\Delta}(F) = \{G \in \Delta: F \cap G = \emptyset, F \cup G \in \Delta\}.$  The \textbf{antistar} (or deletion) of the face $F$ in $\Delta$ is the subcomplex $\Delta - F = \{G \in \Delta: F \nsubseteq G\}.$  

Given a pure simplicial complex $\Delta$ and facets $F,F' \in \Delta$, the \textbf{distance} from $F$ to $F'$ is the length of the shortest path $F = F_0, F_1, \ldots, F_t = F'$ where the $F_i$ are facets and $F_i$ intersects $F_{i+1}$ along a ridge (a codimension-one face) for all $0 \leq i < t.$   The \textbf{diameter} of a pure simplicial complex, denoted $\diam(\Delta)$, is the maximum distance between any two facets in $\Delta$.  

One approach to trying to establish (polynomial) diameter bounds is to study decompositions of simplicial complexes.  Provan and Billera \cite{Billera-Provan} defined a notion of $k$-decomposability for simplicial complexes and showed that $k$-decomposable complexes satisfy nice diameter bounds. 

\begin{definition}{\rm{(\cite[Definition 2.1]{Billera-Provan})}}
Let $\Delta$ be a $(d-1)$-dimensional simplicial complex and let $0 \leq k \leq d-1$.  We say that $\Delta$ is $\mathbf{k}$\textbf{-decomposable} if $\Delta$ is pure and either
\begin{enumerate} 
\item $\Delta$ is a $(d-1)$-simplex, or 
\item there exists a face $\tau \in \Delta$ (called a \textbf{shedding face}) with $\dim(\tau) \leq k$ such that 
\begin{enumerate}
\item $\Delta - \tau$ is $(d-1)$-dimensional and $k$-decomposable, and 
\item $\lk_{\Delta}(\tau)$ is $(d-|\tau|-1)$-dimensional and $k$-decomposable.
\end{enumerate}
\end{enumerate}
\end{definition}

\begin{theorem}{\rm{(\cite[Theorem 2.10]{Billera-Provan})}}
Let $\Delta$ be a $k$-decomposable simplicial complex of dimension $d-1$.  Then $$\diam(\Delta) \leq f_k(\Delta) - {d \choose k+1},$$ where $f_k(\Delta)$ denotes the number of $k$-dimensional faces in $\Delta$.  
\end{theorem}

In particular, a $0$-decomposable complex (also called vertex-decomposable) satisfies the Hirsch bound.  One approach to trying to prove the Hirsch conjecture would be to try to show that any simplicial polytope is vertex-decomposable.  In his thesis, Lockeberg \cite{Lockeberg} constructed a simplicial $4$-polytope on $12$ vertices that is not vertex-decomposable (see also \cite[Proposition 6.3]{Klee-Kleinschmidt}\footnote{There is a small typo in \cite[Proposition 6.3]{Klee-Kleinschmidt}.  The listed facet \texttt{aejk} should be \texttt{aehk} instead.  We found this error by entering the listed polytope into Sage~\cite{sage} and realizing that some of its ridges were contained in a unique facet.}).  Of course, Santos' counterexample to the Hirsch conjecture provides another example of a simplicial polytope that is not vertex-decomposable. 

In addition, Provan and Billera defined a weaker notion of $k$-decomposability that does not require any condition on links but still provides bounds on the diameter of the simplicial complex.

\begin{definition}{\rm{(\cite[Definition 4.2.1]{Billera-Provan})}}
Let $\Delta$ be a $(d-1)$-dimensional simplicial complex and let $0 \leq k \leq d-1$.  We say that $\Delta$ is \textbf{weakly $\mathbf{k}$-decomposable} if $\Delta$ is pure and either 
\begin{enumerate}
\item $\Delta$ is a $(d-1)$-simplex or
\item there exists a face $\tau \in \Delta$ with $\dim(\tau) \leq k$ such that $\Delta-\tau$ is $(d-1)$-dimensional and  weakly $k$-decomposable.
\end{enumerate}
\end{definition}

\begin{theorem}{\rm{(\cite[Theorem 4.2.3]{Billera-Provan})}}
Let $\Delta$ be a weakly $k$-decomposable simplicial complex of dimension $d-1$.  Then $$\diam(\Delta) \leq 2f_k(\Delta).$$
\end{theorem}

Again, we say that a weakly $0$-decomposable complex is weakly vertex-decomposable, abbreviated \emph{wvd}.  Based on the hope that diameters of simplicial polytopes have linear upper bounds, it would be natural to try to prove that any simplicial $d$-polytope is weakly vertex-decomposable.  
In Section~\ref{non-wvd-example}, we will provide a family of  simple transportation polytopes whose polars are not weakly vertex-decomposable.  

\subsection{Transportation polytopes}

Our counterexamples are actually found within the classical family of transportation problems. These are classical polytopes that
play an important role in combinatorial optimization and the theory of networks \cite{YKK-book}. For general notions about
polytopes see~\cite{zieglerbook}.
For fixed vectors $\mathbf{a} = (a_1,\ldots,a_m) \in \mathbb{R}^m$ and $\mathbf{b} = (b_1,\ldots,b_n) \in \mathbb{R}^n$, the \textbf{classical} $\mathbf{m \times n}$\textbf{ transportation polytope} $P(\mathbf{a},\mathbf{b})$ is the collection of all nonnegative matrices $X = (x_{i,j})$ with $\sum_{i=1}^mx_{i,j} = b_j$ for all $1 \leq j \leq n$ and $\sum_{j=1}^nx_{i,j} = a_i$ for all $1 \leq i \leq m$.  The vectors
$\mathbf{a,b}$ are often called the \textbf{margins} of the transportation problem.

There is a natural way to associate a complete bipartite graph $K_{m,n}$ with weighted edges to each matrix $X \in P(\mathbf{a},\mathbf{b})$ by placing a weight of $x_{i,j}$ on the edge $(i,j) \in [m] \times [n]$.  We summarize the properties of transportation polytopes that we will use in the following theorem.  These results and their proofs can be found in \cite{Klee-Witzgall} and \cite[Chapter 6]{YKK-book}.

\begin{theorem}\label{TP-properties}
Let $\mathbf{a} \in \mathbb{R}^m$ and $\mathbf{b} \in \mathbb{R}^n$ with $mn > 4$.
\begin{enumerate}
\item The set $P(\mathbf{a},\mathbf{b})$ is nonempty if and only if $\sum_{i=1}^ma_i = \sum_{j=1}^nb_j.$  
\item The dimension of $P(\mathbf{a},\mathbf{b})$ is $(m-1)(n-1)$.
\item The transportation polytope $P(\mathbf{a},\mathbf{b})$ is nondegenerate (hence simple) if and only if the only nonempty sets $S \subseteq [m]$ and $T \subseteq [n]$ for which $\sum_{i \in S} a_i = \sum_{j \in T} b_j$ are $S = [m]$ and $T = [n]$.  \label{nondegenerate}
\item The set $$F_{p,q} = F_{p,q}(\mathbf{a}, \mathbf{b}) := \{X \in P(\mathbf{a},\mathbf{b}): x_{p,q} = 0\},$$ is a facet of $P(\mathbf{a},\mathbf{b})$ if and only if $a_p + b_q < \sum_{i=1}^m a_i.$
\item The matrix $X \in P(\mathbf{a},\mathbf{b})$ is a vertex of $P(\mathbf{a},\mathbf{b})$ if and only if the edges  $\{(i,j) \in K_{m,n}: x_{i,j} > 0\}$ form a spanning tree of $K_{m,n}$.
\end{enumerate}
\end{theorem}

In Figure \ref{fig:2by4} we show an example of a $2 \times 4$ transportation polytope. This example demonstrates the content of the above theorem and at the same time it demonstrates that the complexes we discuss in the next section are vertex decomposable in dimension smaller than five. On the left side of the figure we show a Schlegel diagram of the 3-dimensional transportation polytope with margins $(2,2,2,2)$ and $(3,5)$. Its vertices are labeled by $2\times 4$ tables, but we only represent the three upper left entries since they determine the rest of the values automatically (see the middle example for the vertex $012$). The right side of the figure shows the dual simplicial complex which is clearly vertex decomposable.

\begin{figure}
\begin{center}
\includegraphics[width=9cm]{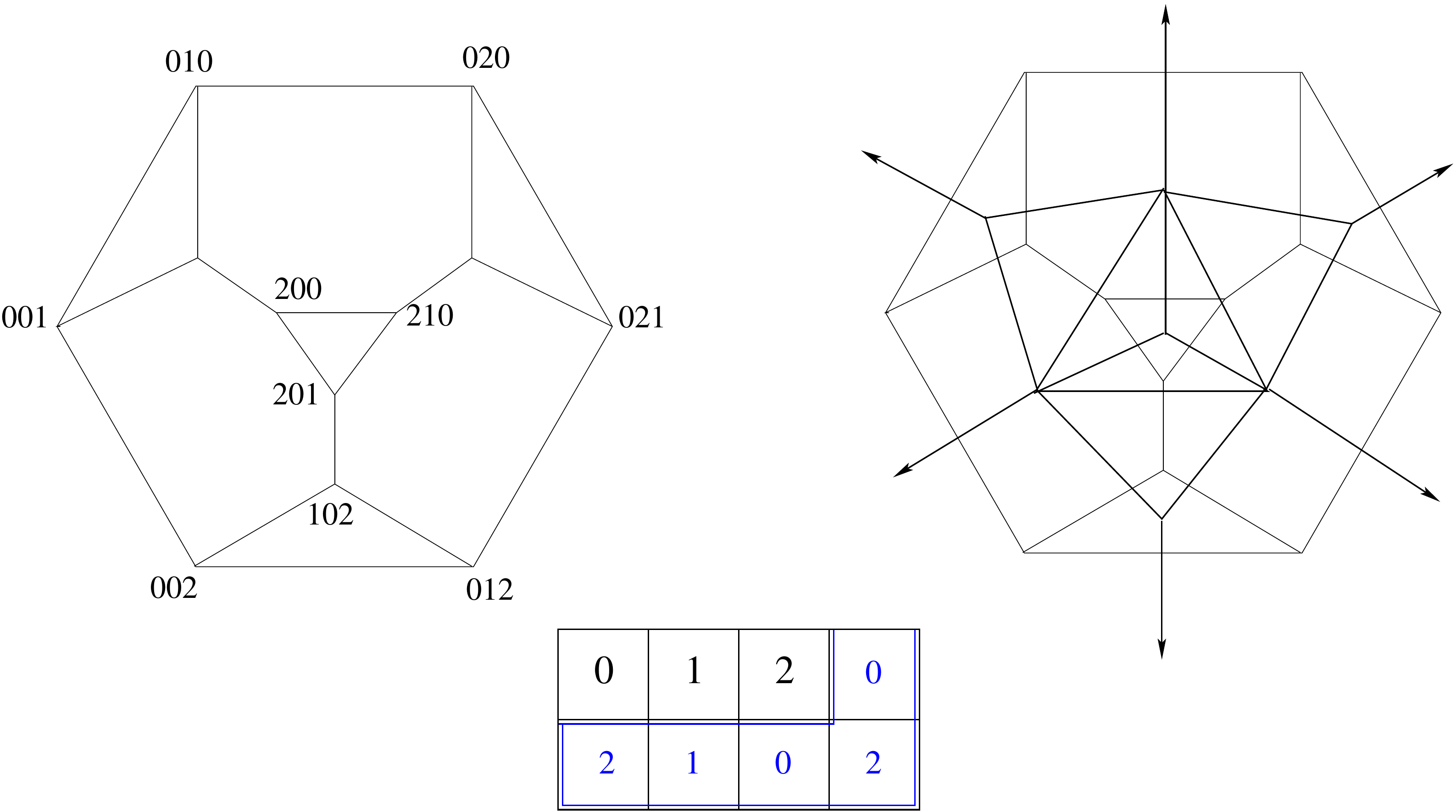}
\caption{A $2 \times 4$ transportation polytope and its polar simplicial complex.}
\label{fig:2by4}
\end{center}
\end{figure}

 It is worth remarking to the reader that transportation polytopes have been heavily studied regarding the diameter of their graphs or 1-skeleton. Unlike the present paper, in most of the literature on the subject, paths move from vertex to vertex along the edges of the polytope instead of moving from facet to facet across ridges, and the Hirsch bound takes the form of $n-d$ where $n$ is the number of facets, instead of vertices for the simplicial set up of this paper.  The best bound for the diameter of the graph of a general transportation polytopes is linear, but still not equal to the Hirsch bound (see \cite{brightwelletal,kimsantossurvey2010}). For our purposes here the most relevant result is about the diameter of $2 \times p$ transportation problems because our counterexamples are polars of those polytopes
 (a result independently obtained by L. Stougie).
  
\begin{theorem}{\rm{(\cite[Theorem 3.5.1]{kim:thesis:2010}) }} \label{kim-bound}
Let $P$  be a classical transportation polytope of size $p\times 2$ with $n\leq 2p$ facets. Then, the dimension of $P$ is $d=p-2$ and the diameter of $P$ is at most $n-d$, thus $P$ satisfies the Hirsch conjecture.
\end{theorem}
 
Once more we stress that  the bounds on the diameters of the graphs of transportation polytopes are equivalent to the simplicial diameter for the polars of transportation polytopes.  Thus the result above is quite relevant to this paper, on the other hand, although the diameters for the \emph{graphs} of the polars of transportation polytopes  were proved to satisfy the Hirsch conjecture in \cite{dualtransp1}, those results have no direct relation to our simplicial investigations.


\section{Examples of non-wvd simplicial polytopes}\label{non-wvd-example}

Now we are ready to present a family of $d$-dimensional transportation polytopes $\Delta_d$ for all $d \geq 5$ whose polar (simplicial) polytopes are not weakly vertex-decomposable.  We must consider two cases based on the parity of $d$. 

\begin{theorem}\label{counterexample}
For all $m \geq 3$, let $\Delta_{2m}$ be the simplicial polytope polar to the $2 \times 2m+1$ transportation polytope 
$P(\mathbf{a},\mathbf{b})$ with margins $\mathbf{a} = (2m+1,2m+1)$ and $\mathbf{b} = (2,2,\ldots,2)$.  Then 
$\Delta_{2m}$ is not weakly vertex-decomposable. 
\end{theorem}

\begin{proof}
Let $u_i$ (respectively $v_i$) denote the vertex in $\Delta_{2n}$ corresponding to the facet $F_{1,i}$ (respectively $F_{2,i}$ ) of $P(\mathbf{a},\mathbf{b})$.  Let $U = \{u_1,\ldots,u_{2m+1}\}$ and $V = \{v_1,\ldots,v_{2m+1}\}$.  We claim that the facets of $\Delta_{2m}$ are precisely those sets of the form $A \cup B$ where 
\begin{itemize}
\item $A \subseteq U$, 
\item $B \subseteq V$, 
\item $|A| = |B| = m$, and 
\item $A \cup B$ contains at most one element from each set $\{u_j,v_j\}$. 
\end{itemize}

Any facet of $\Delta_{2m}$ can be decomposed as $A \cup B$ with $A \subseteq U$ and $B \subseteq V$.  Since any spanning tree of $K_{2,2m+1}$ has $2m+2$ edges, any facet of $\Delta_{2m}$ has $2(2m+1) - (2m+2) = 2m$ vertices. Suppose that there is a facet $A \cup B$ of $\Delta_{2m}$ with $|A| > m$.  We may assume without loss of generality that $u_1,\ldots,u_{m+1} \in A$.  This means there is a matrix $X \in P(\mathbf{a},\mathbf{b})$ with $x_{1,1},x_{1,2},\ldots,x_{1,m+1} = 0$.  Thus $x_{2,1} = x_{2,2} = \cdots = x_{2,m+1} = 2$ and the sum of the elements in the second row of $X$ exceeds $2m+1$.  Similarly, no facet of $\Delta_{2m}$ can contain both $u_j$ and $v_j$ since $P(\mathbf{a},\mathbf{b})$ does not contain a matrix in which $x_{1,j} = x_{2,j} = 0$. 

Suppose that $\Delta_{2m}$ is weakly vertex-decomposable and its vertices can be shed in the order $z_1,z_2,z_3,\ldots, z_t$.  We will show that the complex obtained from $\Delta_{2m}$ by removing either $z_1$ and $z_2$ or $z_1$, $z_2$ and $z_3$ is not pure.  By the pigeonhole principle, two of the vertices among $\{z_1,z_2,z_3\}$ come from either $U$ or $V$, and we may assume without loss of generality that these two vertices come from $U$.  Further, since the symmetric group $\mathfrak{S}_{2m+1}$ acts transitively on the columns of the $2 \times (2m+1)$ contingency table defining $P(\mathbf{a}, \mathbf{b})$, we need only consider two possibilities: either $\{z_1,z_2\} = \{u_1,u_2\}$ or $\{z_1,z_2,z_3\} = \{u_1,v,u_2\}$ for some $v \in \{v_1,v_2,v_3\}$.  

In the former case, let $\Gamma_1$ be the simplicial complex obtained from $\Delta_{2m}$ by removing vertices $z_1$ and $z_2$, and consider the following facets of $\Delta_{2m}$:
\begin{eqnarray*}
F &=& \{u_1,u_3,u_4,\ldots,u_{m+1}, v_{m+2},\ldots,v_{2m+1}\}, \\
F' &=& \{u_2,u_3,u_4,\ldots,u_{m+1}, v_{m+2},\ldots,v_{2m+1}\}, \text{ and } \\
G &=& \{v_2,v_3,\ldots,v_{m+1}, u_{m+2}, \ldots,u_{2m+1}\}.
\end{eqnarray*}

Then $\Gamma_1$ is $(2m-1)$-dimensional since it contains $G$ as a facet, but $F - \{u_1\} = F' - \{u_2\}$ is a $(2m-2)$-face of $\Gamma_1$ that is not contained in a $(2m-1)$-face.  Thus $\Gamma_1$ is not pure.  The following partially filled contingency table shows that $F$ and $F'$ are the only facets of $\Delta_{2m}$ that contain $F-\{u_1\} = F'-\{u_2\}$.   

\begin{center}
\begin{tabular}{p{.75cm}p{.75cm}p{.75cm}p{.75cm}p{.75cm}p{.75cm}p{.75cm}p{.75cm}p{1cm}p{.75cm}}
\multicolumn{1}{c}{1} & \multicolumn{1}{c}{2} & \multicolumn{1}{c}{3} & \multicolumn{1}{c}{$\cdots$} & \multicolumn{1}{c}{$m$} & \multicolumn{1}{c}{$m+1$} &\multicolumn{1}{c}{ $m+2$} & \multicolumn{1}{c}{$\cdots$} & \multicolumn{1}{c}{$2m+1$} & \\ \cline{1-9}
\multicolumn{1}{|c}{$x_{1,1}$}  & \multicolumn{1}{|c}{$x_{1,2}$}  & \multicolumn{1}{|c}{$0$}  &    \multicolumn{1}{|c}{$\cdots$}      &  \multicolumn{1}{|c}{$0$}   &   \multicolumn{1}{|c}{$0$}    &   \multicolumn{1}{|c}{$2$}    &    \multicolumn{1}{|c}{$\cdots$}      &    \multicolumn{1}{|c}{$2$}    & \multicolumn{1}{|c}{$2m+1$} \\ \cline{1-9}
\multicolumn{1}{|c}{$x_{2,1}$}  & \multicolumn{1}{|c}{$x_{2,2}$}  & \multicolumn{1}{|c}{$2$}  &   \multicolumn{1}{|c}{$\cdots$}       & \multicolumn{1}{|c}{$2$}    &  \multicolumn{1}{|c}{$2$}     &   \multicolumn{1}{|c}{$0$}    &    \multicolumn{1}{|c}{$\cdots$}      &  \multicolumn{1}{|c}{$0$}      & \multicolumn{1}{|c}{$2m+1$} \\ \cline{1-9}
\multicolumn{1}{c}{$2$} & \multicolumn{1}{c}{$2$} & \multicolumn{1}{c}{$2$} & \multicolumn{1}{c}{$2$} & \multicolumn{1}{c}{$2$} & \multicolumn{1}{c}{$2$} & \multicolumn{1}{c}{$2$} & \multicolumn{1}{c}{$2$} & \multicolumn{1}{c}{$2$} &
  
\end{tabular}
\end{center}

In the latter case, let $\Gamma_2$ be the simplicial complex obtained from $\Delta_{2m}$ by removing vertices $z_1$, $z_2$, and $z_3$.  Again, $\Gamma_2$ is $(2n-1)$-dimensional since it contains the facet
\begin{eqnarray*}
G' = \{v_1,v_2,\ldots,v_{m+1},u_{m+2},\ldots,u_{2m+1}\} - \{v\},
\end{eqnarray*}
of $\Delta_{2m}$, but $\Gamma_2$ is not pure since the face $F - \{u_1\} = F' - \{u_2\}$ is not contained in any $(2m-1)$-face of $\Gamma_2$.
\end{proof}

A similar construction provides odd-dimensional transportation polytopes whose polars are simple and non-wvd.  Unfortunately, we must tweak the construction presented in Theorem~\ref{counterexample} because the margins $\mathbf{a}=(2m,2m)$ and $\mathbf{b}=(2,2,\ldots,2)$ yield a degenerate transportation polytope by Theorem~\ref{TP-properties}.\ref{nondegenerate}.  The proof of the following theorem is identical to that of Theorem~\ref{counterexample}.

\begin{theorem}\label{counterexample-even}
For all $m \geq 3$, let $\Delta_{2m-1}$ be the simplicial polytope dual to the $2 \times 2m$ transportation polytope $P(\mathbf{a},\mathbf{b})$ with $\mathbf{a} = (2m-1,2m+1)$ and $\mathbf{b} = (2,2,\ldots,2)$.  Then $\Delta_{2m-1}$ is not weakly vertex-decomposable.
\end{theorem}

Despite the fact that the polars to these transportation polytopes are not weakly vertex-decomposable, we see from Theorem~\ref{kim-bound} that the polytopes $P(\mathbf{a},\mathbf{b})$ in Theorems~\ref{counterexample} and \ref{counterexample-even} still satisfy the Hirsch bound. Note that our counterexamples yield an infinite family of non-vertex decomposable polytopes.


\section*{Acknowledgements}
The first author was partially supported by NSF grant DMS-0914107, the second author was supported by NSF VIGRE grant DMS-0636297.  We truly grateful to Scott Provan for sharing his thoughts on weak vertex-decomposability of simplicial polytopes and to Peter Kleinschmidt for helping us find the original text of Lockeberg's thesis \cite{Lockeberg} so that we could resolve the typographical error in \cite[Proposition 6.3]{Klee-Kleinschmidt}.  We are also grateful to Yvonne Kemper for helpful conversations during the cultivation of the ideas presented in this paper.



\bibliography{wvd-biblio}
\bibliographystyle{plain}

\end{document}